\documentclass{article}
\usepackage{t1enc,amsfonts,latexsym,amsmath,amsthm,verbatim,amssymb,graphics,mathdots,pstool,color,wrapfig}
\usepackage{graphicx}
\usepackage{t1enc,amsmath,amsfonts,latexsym,verbatim,amssymb,graphics,mathdots,pstool,color}
\usepackage[utf8]{inputenc} 
\usepackage{pgfplots}
\usepackage{geometry}
\usepackage[showonlyrefs]{mathtools}
\usepackage{courier}
\usepackage{listings}
\usepackage{wrapfig}
\usepackage{mathrsfs}
\usepackage[numbers]{natbib}

\newcommand{\rank}{\mathsf{rank}}
\newcommand{\card}{\mathsf{card}}
\renewcommand{\Re}{\mathsf{Re}}
\DeclareMathOperator*{\argmin}{arg\,min ~}

\newcommand{\fro}[1]{\left\| #1\right\|^2}
\newcommand{\scal}[1]{\left \langle #1 \right \rangle}
\newcommand{\para}[1]{\left( #1\right)}
\newcommand{\R}{\mathbb{R}}
\newcommand{\C}{\mathbb{C}}

\newcommand{\m}{\mathbb{M}}

\newcommand{\V}{\mathcal{V}}
\newcommand{\W}{\mathcal{W}}
\newcommand{\U}{\mathcal{U}}

\newcommand{\J}{\mathscr{J}}

\renewcommand{\S}{\mathcal{S}}
\newcommand{\Q}{\mathcal{Q}}

\renewcommand{\H}{\mathcal{H}}

\newtheorem{theorem}{Theorem}[section]
\newtheorem{proposition}[theorem]{Proposition}

\newtheorem{corollary}[theorem]{Corollary}

\begin{document}

\title{On Convex Envelopes and Regularization of Non-Convex Functionals without moving Global Minima.}

\author{Marcus Carlsson\thanks{Centre for Mathematical Sciences , Lund University,  {mc@maths.lth.se}}
}
\date{}

\maketitle

\begin{abstract}
We provide theory for the computation of convex envelopes of non-convex functionals including an $\ell^2$-term, and use these to suggest a method for regularizing a more general set of problems. The applications are particularly aimed at compressed sensing and low rank recovery problems but the theory relies on results which potentially could be useful also for other types of non-convex problems. For optimization problems where the $\ell^2$-term contains a singular matrix we prove that the regularizations never move the global minima. This result in turn relies on a theorem concerning the structure of convex envelopes which is interesting in its own right. It says that at any point where the convex envelope does not touch the non-convex functional we necessarily have a direction in which the convex envelope is affine.
\end{abstract}

\section{Introduction}

This article is a compressed and improved version of \cite{carlsson2016arxiv}, which contains more information and potentially more errors. The present work is the extension of a chain of ideas with its roots in compressed sensing. $\ell^1-\ell^2$-minimization tricks have a long history and got renewed attention with the work of Donoho, Cand\'{e}s and Tao among others \cite{chen2001atomic,donoho2006most,candes2005decoding}. In the same spirit the nuclear norm minimization strategy was investigated by Fazel and coworkers \cite{fazel2002matrix,recht2010guaranteed} and in both cases it was shown that these methods yield perfect reconstructions in the case of no noise. However, in realistic scenarios these results often do not apply and moreover there is of course noise, in which case the methods come with a (sometimes severe) bias. Moreover they are slow since one needs to find an appropriate value of involved penalty parameters.

Due to such issues there is a wealth of non-convex variations to replace $\ell^1/$nuclear norm in the area of compressed sensing, we refer to \cite{ourselves} for a survey. Two fairly recent contributions in this vein is the work by Carl Olsson and coworkers \cite{larsson2016convex} as well as by Gilles Aubert and coworkers \cite{soubies2015continuous}. The former paper deals with non-convex matrix minimization problems with subspace constraints, the latter with sparse reconstructions, and in particular the latter shows that the concrete regularizer considered there has the desirable property of not moving global minima. In this paper we find a unifying framework and show that all these penalties are particular cases of the so called ``proximal hull'' or ``quadratic envelope''. We systematically study this as a regularizer and in particular we lift the result of Aubert et al.~to a general context. In order to do so we provide new results on the structure of lower semi-continuous (abbreviated l.s.c.) convex envelopes which are interesting in their own right. More precisely we show that whenever a l.s.c.~convex envelope is not in touch with the function that generates it, then it necessarily has a direction in which it is affine linear.

\section{Outline and motivation}\label{sec_outline}

We develop methods to compute the lower semi-continuous convex envelope of functionals of the form \begin{equation}\label{pgf}f(x)+\frac{1}{2}\|x-d\|^2_2,\end{equation} and show that this is of the form $\Q(f)(x)+\frac{1}{2}\|x-d\|^2_2$, where $\Q(f)$ is the proximal hull or ``quadratic envelope'', as we shall call it. Here $x$ can be in any separable Hilbert space but $f$ needs to be such that the global minimization of \eqref{pgf} is computable. The practical applications of $\Q(f)$ pertains to optimization of \eqref{pgf} with additional constraints, as well as unconstrained optimization of \begin{equation}\label{pgf1}f(x)+\frac{1}{2}\|Ax-d\|^2_2,\end{equation} where $A$ is a linear operator.

To introduce the main ideas behind this work we consider two concrete problems. A multitude of applications can be posed mathematically as finding the lowest rank matrix $X$ satisfying some equation ${A}(X)=d$, where $A$ is a linear operator and $d$ is a measurement (see e.g.~\cite{recht2010guaranteed,tseng2010approximation}). Usually the measurement $d$ is not perfect so in practice one wishes to find the minimum rank given some accepted error; $\|A(X)-d\|\leq \rho$. The dual formulation of this problem is \begin{equation}\label{i8}\argmin_{X}\lambda\rank(X)+\|A(X)-d\|^2,\end{equation}\begin{wrapfigure}{r}{0.5\textwidth}
     \includegraphics[width=0.5\textwidth]{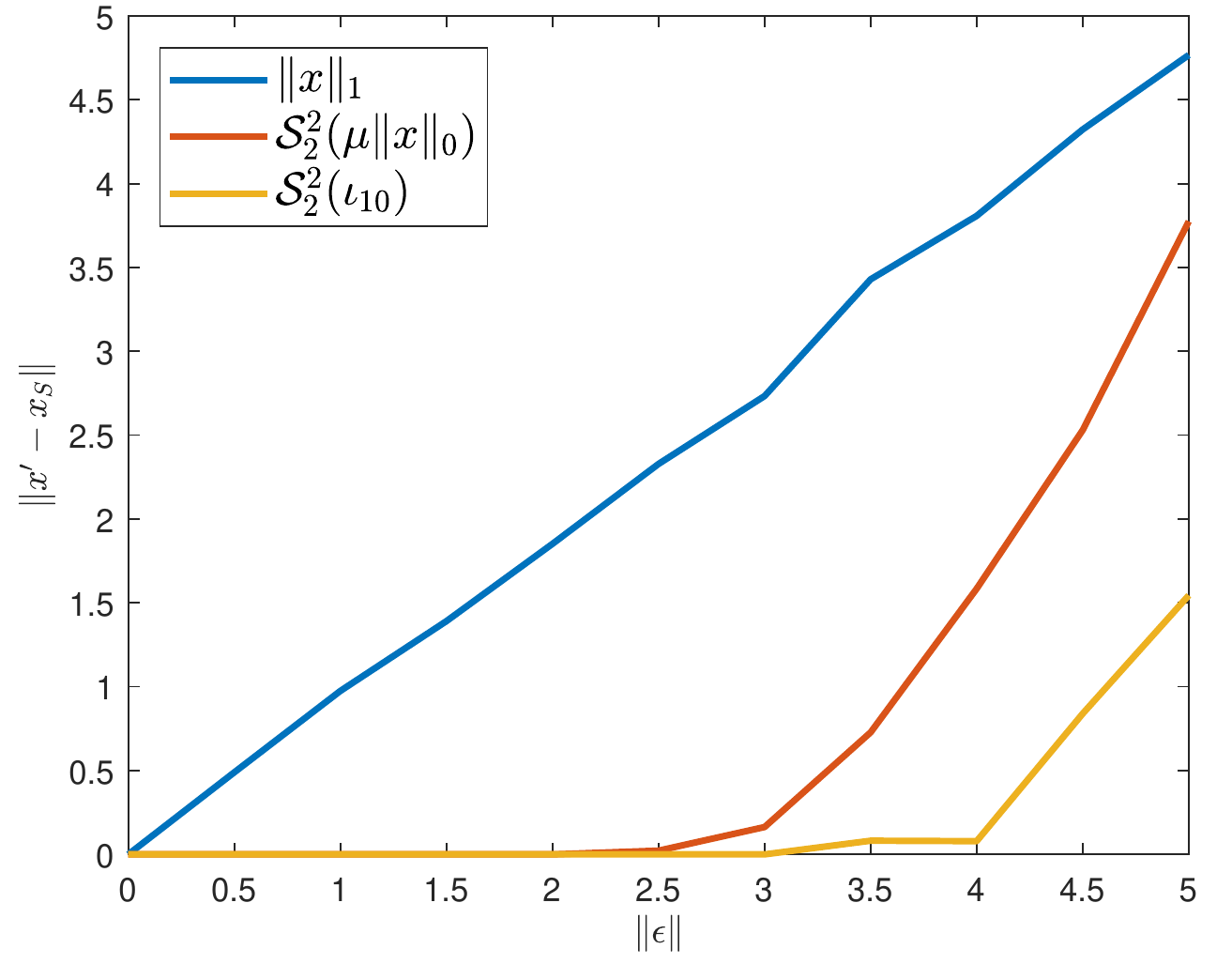}
 \caption{Reguarlizing \eqref{t4} by $\Q(\card)$ and $\Q(\iota_{10})$ finds the oracle solution up to noise levels of around $\|\epsilon\|=3$ and $4$ (roughly 30\% of $\|d\|$) whereas $\ell^1$-regularization only finds this solution with no noise.}\label{f4}
\end{wrapfigure} where $\lambda$ is a parameter. However, the functional $\rank(X)$ is non-convex and highly discontinuous, so the problem can not be solved as stated (in general). It can be solved for the case $A=I$ but the problem is still hard when combined with additional priors, see e.g.~Section 1.1 in \cite{larsson2016convex} for an overview and  applications in signal processing and imaging. 

Due to the problematic nature of $\rank(X)$ it has become popular to replace $\rank(X)$ with the nuclear norm of $X$. However, $\rank (X)$ and the nuclear norm are quite far apart and the method leads to a bias in the solution, which led the authors of  \cite{larsson2016convex} to suggest working instead with the convex envelope of $\rank(X)+\frac{1}{2}\|X-D\|^2_F$ for which they obtained an explicit expression. They also provided the convex envelope when $\rank(X)$ is replaced by the indicator functional of the set  $\{X:~\rank(X)\leq K\}$,
in order to treat problems where a matrix of a fixed rank is sought, and this convex envelope was further studied in \cite{andersson2016convex}.

Independently, convex envelopes was used in \cite{soubies2015continuous} to suggest a regularizer to functionals of the type \begin{equation}\label{t5} \|x\|_0+\frac{1}{2}\|Ax-d\|^2_2,\quad x\in\R^n,\end{equation} which is usually dealt with by replacing $\|x\|_0$ by $\lambda\|x\|_1$. The main contribution of their work is to show that their regularizer does not move global minima. A common misconception is that the same holds for $\ell^1$-methods, which is true only if there is no noise \cite{candes2006robust}. In the presence of noise the estimates for $\ell^1$-methods are rather poor and \cite{soubies2015continuous} is the first framework which allows for regularization without moving minima in a more realistic scenario.

This paper presents a unified approach to this circle of ideas by connecting them with the 
 ``quadratic envelope'' $\Q(f)$. We also extend the findings of \cite{soubies2015continuous} to any problem of the form \eqref{pgf1} as long as $\Q(f)$ is computable. An expanded version of this article is found in \cite{carlsson2016arxiv} which contains a long list of instances where $\Q(f)$ is computable.

In particular $\Q(f)$ is computable for $\iota_K$; the indicator functional for $\{x\in\C^n:~\|x\|_0\leq K\}$. As a proof of concept, we compare performance of \eqref{t5} with $\|x\|_0$ replaced by $\lambda \|x\|_1$, with $\Q(\card)$ and with $\Q(\iota_K)$ (where $\card(x)=\|x\|_0$). We use a $100\times 200$ matrix $A$ and minimize the regularized version of \eqref{t5} for $d$ of the form $Ax_0+\epsilon$, where $x_0$ has cardinality 10 and $\epsilon$ takes on various levels of noise. As noted in \cite{ctr} the best one can hope for is then to recover the so called ``oracle solution'' $x_S$ (obtained if an oracle a priori revealed the correct support). As Figure \ref{f4} shows both $\Q(\card)$ and $\Q(\iota_{10})$ outperform $\ell^1$ and finds the oracle solution for fairly large levels of noise. Also $\Q(\iota_{10})$ beats $\Q(\card)$, which is no surprise since it contains additional information about the problem built into it, and demonstrates the versatility of the new $\Q$-transform. The article \cite{ourselves} contain much more information about this particular case.

We now outline the main contributions of this paper in greater detail. Consider any functional of the form
\begin{equation}\label{t553} f(x)+\frac{\gamma}{2}\|x-d\|^2_\V\end{equation}
where $\gamma>0$ is a parameter, $\V$ is an arbitrary separable Hilbert space and $f$ a non-negative functional on $\V$. 
In Section \ref{son} we introduce the transform $\Q_\gamma$ and show that the l.s.c.~convex envelope of the functional in \eqref{t553} is
\begin{equation}\label{t55}\Q_\gamma(f)(x)+\frac{\gamma}{2}\|x-d\|^2_\V.\end{equation} In order for $\Q_\gamma(f)$ to be computable, the global minimization of \eqref{t553} needs to be solvable, and hence the problem of minimizing \eqref{t553} in itself is not an instance where the $\Q_\gamma$-transform is useful. However, it is useful for finding global minimizers of \eqref{t553} in combination with additional prior restrictions. To illustrate, consider the problem
\begin{equation}\label{restrictedproblem1}\argmin_{x\in\H} f(x)+\frac{1}{2}\|x-d\|^2,\end{equation}
where $\H$ is a closed convex subset of $\V$, and suppose we are unable to find a closed form solution. Upon replacing \eqref{restrictedproblem1} with
\begin{equation}\label{restrictedproblem3}\argmin_{x\in\H} \Q_\gamma(f)(x)+\frac{1}{2}\|x-d\|^2,\end{equation} for some fixed
$\gamma\leq 1$, we obtain a convex problem which can be solved. However, even for $\gamma=1$ it is possible that \eqref{restrictedproblem1} and \eqref{restrictedproblem3} have different solutions, despite the functional in \eqref{restrictedproblem3} being the l.s.c.~convex envelope of the one in \eqref{restrictedproblem1}. The rationale behind replacing \eqref{restrictedproblem1} by \eqref{restrictedproblem3} is pragmatical; since the latter is convex the solution may be found using convex optimization routines. This may seem ad hoc but we remind the reader that replacing e.g.~$\|x\|_0$ by $\|x\|_1$ or $\rank(X)$ by the nuclear norm $\|X\|_1$ has had a substantial impact, and that for these concrete cases the modification $\Q_\gamma(f)$ is much closer to the original functional $f$ (which leads to a better performance as an estimator, see the numerical sections of \cite{ourselves,larsson2016convex}). A reason for this is that $\Q_\gamma(f)$ has the desirable feature that $\Q_\gamma(f)(x)=f(x)$ often holds, and since \eqref{restrictedproblem3} is a convex problem below the original problem \eqref{restrictedproblem1}, it is easy to see that a minimum $\hat x$ to \eqref{restrictedproblem3} is the solution to \eqref{restrictedproblem1} whenever $\Q_\gamma(f)(\hat x)=f(\hat x)$. This is highlighted in Figure \ref{fig2intro} where the two problems have the same solution. More information and examples on this type of problems is found in Part II of \cite{carlsson2016arxiv}.

\begin{figure}
\centering
\includegraphics[width=0.48\linewidth]{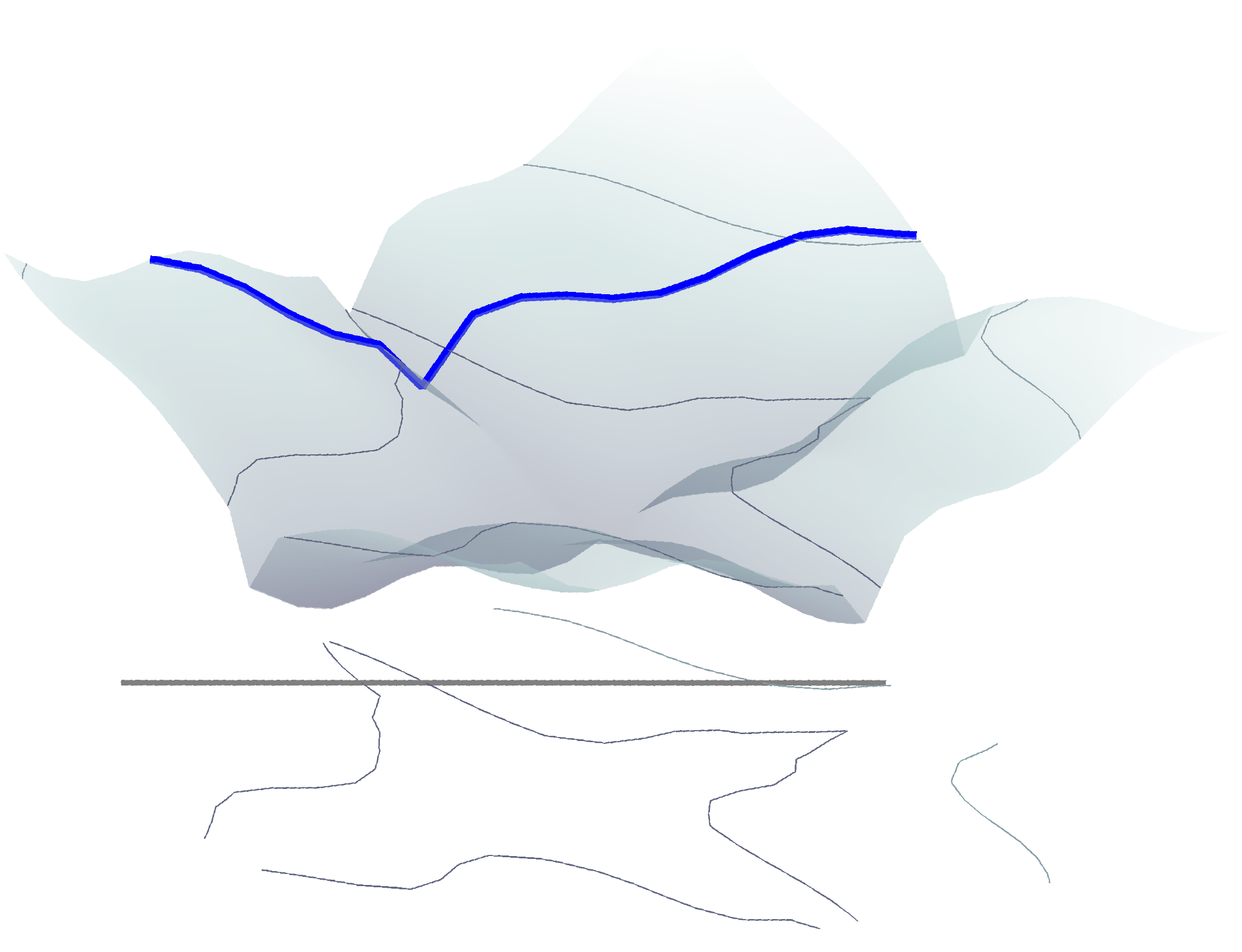}
\includegraphics[width=0.48\linewidth]{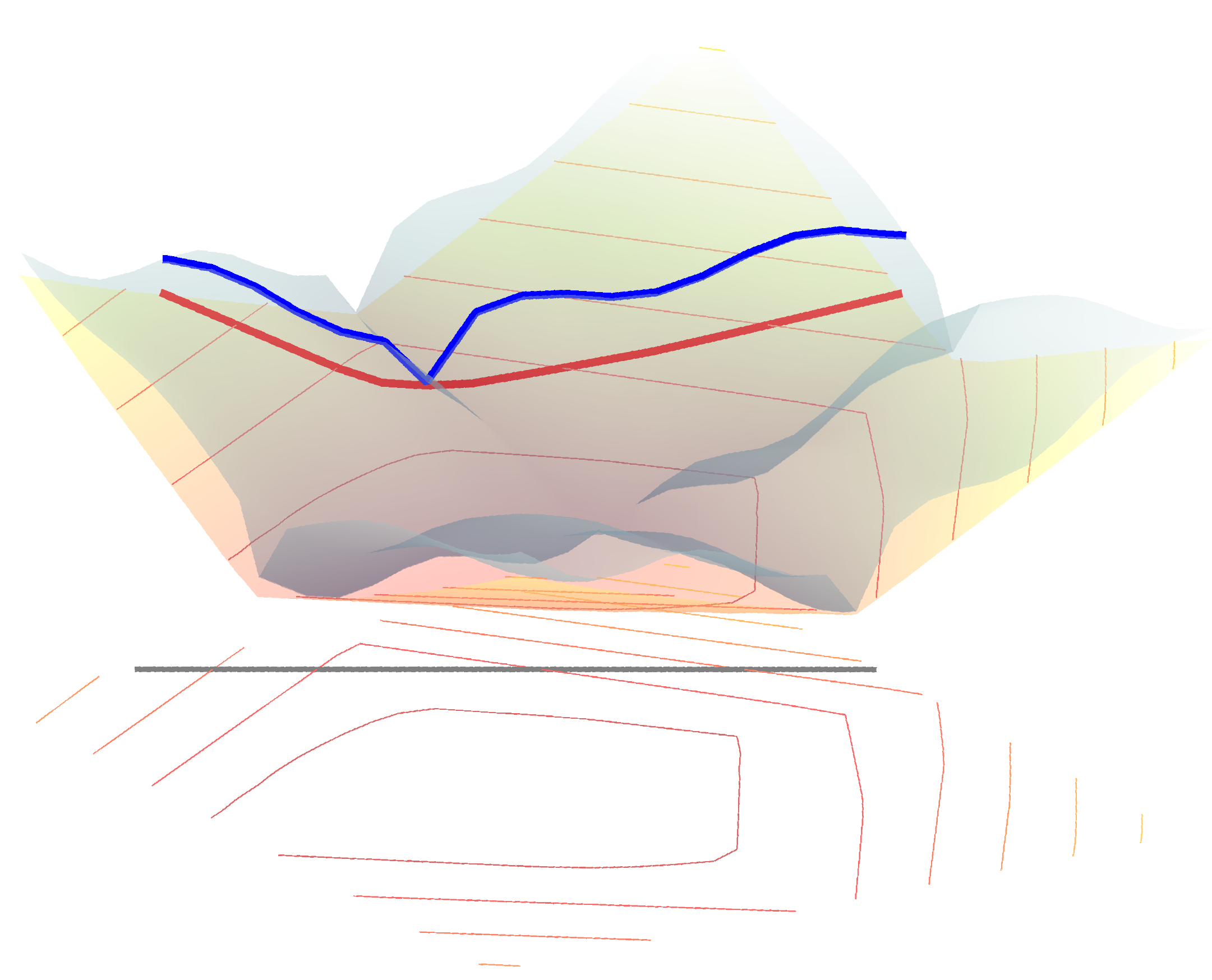}
\caption{Illustration of a non-convex optimization problem with linear constraints. The left panel shows a non-convex functional along with its level sets. The gray line represents the subspace we are interested in, and the blue curve the values of the functional restricted to the subspace. The right panel shows the same setup, but here the convex envelope is shown as well in orange/yellow. The values of the convex envelope over the subspace is shown in the red curve. In this case, the minima of blue and red function coincide.}
\label{fig2intro}
\end{figure}

 In Section \ref{Sreg} we consider regularization of functionals like \eqref{i8} and \eqref{t5}, or more generally \begin{equation}\label{t4i}f(x)+\frac{1}{2}\fro{A x-d}_\W\end{equation} for arbitrary non-negative $f$, where $A:\V\rightarrow \W$ is a linear operator between separable Hilbert spaces. We assume that $\V$ is such that $\Q_\gamma(f)$ is computable and that the convex envelope of \eqref{t4i} is untractable. We propose to use as regularizer the function $\Q_\gamma(f)$, i.e. we will study the relationship between minimizers of \eqref{t4i} and those of \begin{equation}\label{t4modi}\Q_\gamma(f)(x)+\frac{1}{2}\fro{A x-d}_\W.\end{equation}
Since it often holds that $\Q_\gamma(f)(x)=f(x)$, we again see that a global minimizer of \eqref{t4modi} for which this is the case must also be a global minimizer of \eqref{t4i}, in view of the inequality $\Q_\gamma(f)\leq f$ (shown in Section \ref{son}). The parameter $\gamma$ now becomes a useful tool as it tunes the curvature of $\Q_\gamma(f)$ and we pause to illustrate its role by considering a toy problem in one variable; see Figure \ref{fig3intro}. We let $|x|_0$ be the function equalling 1 on $\R\setminus\{0\}$ and zero at $x=0$. In red we see the functional $|x|_0+\frac{1}{2}|x-1|^2$ (which is a particular case of both \eqref{i8} and \eqref{t5} in dimension 1, the matrix $A$ is here the number 1), in blue its convex envelope and in pink the $\ell^1$ convex relaxation $|x|+\frac{1}{2}|x-1|^2$. Clearly the global minimum of the red and blue coincide, but the global minimum of the $\ell^1$-relaxation is different.
For \eqref{t4modi} we have two options, either $|A|^2>\gamma$ or $|A|^2<\gamma$. The regularizer \eqref{t4modi} is illustrated in black for these two cases in Figure \ref{fig3intro}. The circles represent global minima of the respective functions. In the case $|A|^2>\gamma$ we see that \eqref{t4modi} is a \textit{convex} minorant of \eqref{t4i} whose global minima (for this choice of parameters) is equal to that of \eqref{t4modi}. In the case $|A|^2<\gamma$, \eqref{t4modi} is no longer convex but the local minima of \eqref{t4modi} are also minima of \eqref{t4i}, and \eqref{t4modi} has fewer local minima. In particular the global minima coincide. The main point of the paper is loosely that the general behavior is the same.

\begin{figure}
\centering
\includegraphics[width=0.48\linewidth]{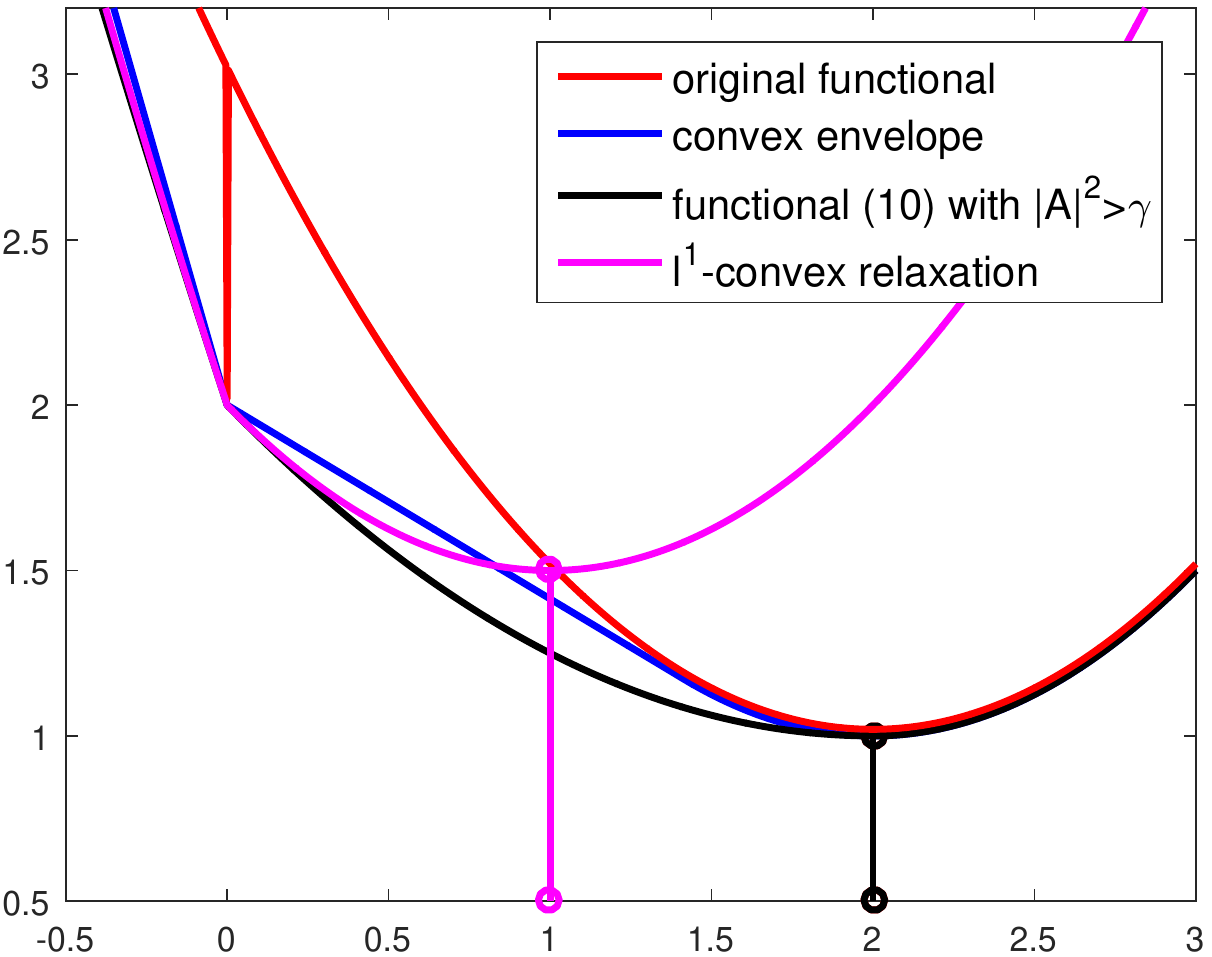}
\includegraphics[width=0.48\linewidth]{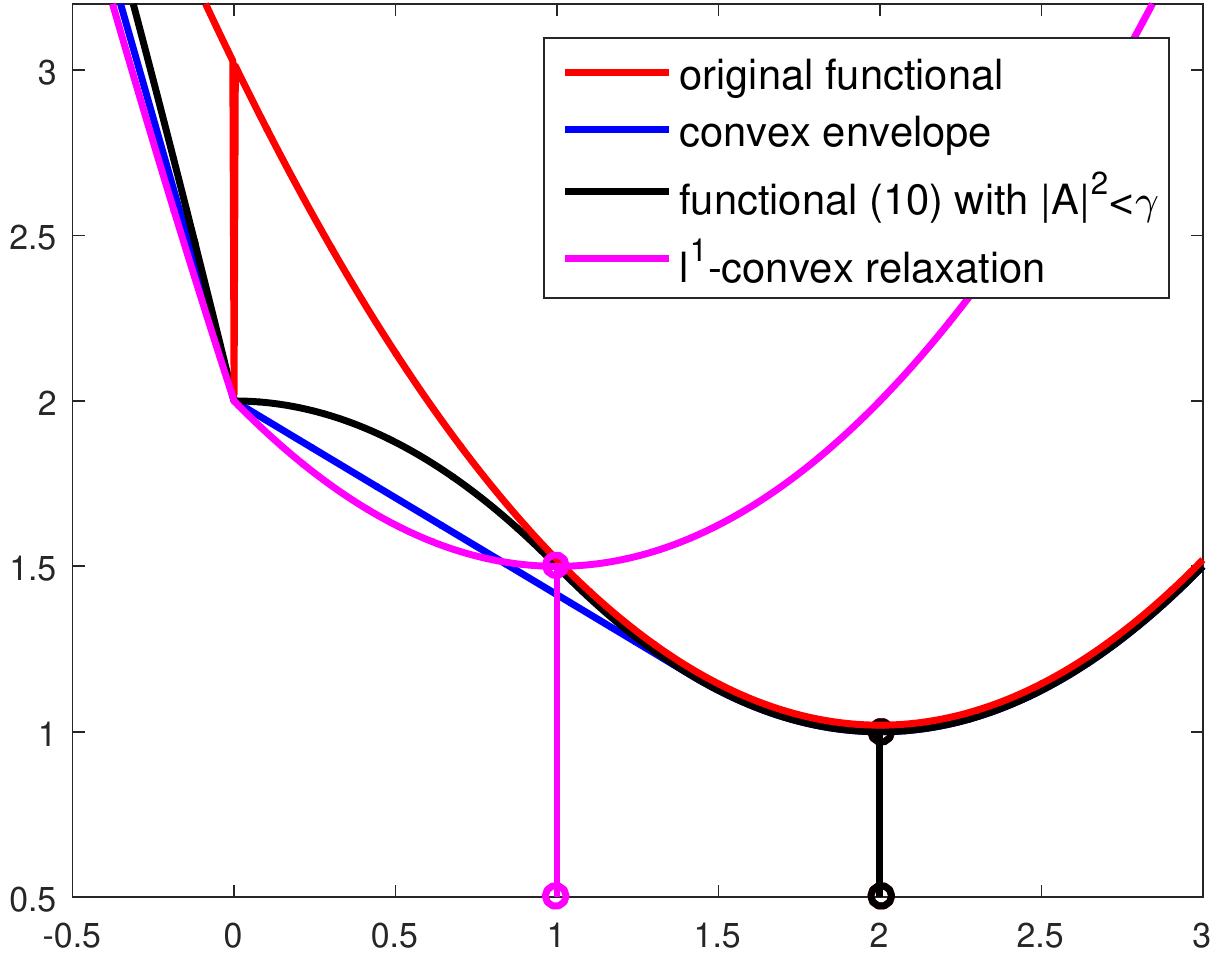}
\caption{The black curve shows two regularizations of the red curve, for different levels of $\gamma$.}
\label{fig3intro}
\end{figure}

In Section \ref{secunderestimate} we generalize the situation in Figure \ref{fig3intro} (left) and assume that $\gamma$ satisfies $A^*A\succcurlyeq \gamma I$, i.e. that \begin{equation}\label{case 1}\|Ax\|^2> \gamma \|x\|^2.\end{equation} For such choice of $\gamma$ we prove that the functional \eqref{t4modi} is a convex functional below \eqref{t4i} and hence minimization of \eqref{t4modi} will produce a minimizer which, although not necessarily equal to the minimizer of the original problem, potentially is closer than that obtained by other convex relaxation methods. 

For the problem \eqref{t5} $A$ is usually a matrix with a large kernel which rules out the above approach. In Section \ref{secoverestimate} we consider the case \begin{equation}\label{case 2}\|A\|^2\leq \gamma,\end{equation} generalizing the situation in the right picture of Figure \ref{fig3intro}. We can then show that \eqref{t4modi} is a continuous (but not convex) functional with the following desirable properties: \begin{itemize}\item[$i)$] \eqref{t4modi} lies between \eqref{t4i} and its l.s.c. convex envelope, \item[$ii)$] any local minimizer of \eqref{t4modi} is a local minimizer of \eqref{t4i}, \item[$iii)$] the global minimizers of \eqref{t4modi} and \eqref{t4i} coincide. \end{itemize}

These findings in turn rely on general results about l.s.c.~convex envelopes which we provide in Section \ref{finer}. The computation of the l.s.c.~convex envelope of $f(x)+\frac{\gamma}{2}\|x\|^2$ can be thought of as stretching a plastic foil from below onto the graph of $f(x)+\frac{\gamma}{2}\|x\|^2$ (see Figure \ref{fig2intro}). Consider a point $x$ where the plastic foil is not in contact with the graph, i.e. where $\Q_\gamma(f)(x)<f(x)$. It is intuitively obvious that the plastic foil, i.e. the graph of $\Q_\gamma(f)(x)+\frac{\gamma}{2}\|x\|^2$, has some direction in which it is affine linear and thus $\Q_\gamma(f)$ should have some direction in which the curvature is $-\gamma$. This is surprisingly difficult to show and despite the wealth of results on l.s.c.~convex envelopes it is not found in any standard reference on the topic. The statement is shown in the PhD-thesis \cite{lucet} for the finite dimensional case. Here we provide a proof is based on an extension of Milman's theorem due to Arne Br\o ndsted \cite{brondsted1966milman} in a short note from 1966. 

The final Section \ref{semi} is more practical in nature. Critical points of \eqref{t4modi} can be found using the forward-backward splitting method (FBS), given that $\Q_\gamma(f)$ is ``semi-algebraic'', as was shown in \cite{attouch2013convergence}. To simplify verification of when $\Q_\gamma(f)$ is semi-algebraic we show in Section \ref{semi} that this is true as long as $f$ itself is semi-algebraic. Further tools to compute $\Q_\gamma(f)$ as well as related proximal operators are found in \cite{carlsson2016arxiv}.

\section{The quadratic envelope}\label{son}

Let $\V$ be a separable Hilbert space over $\R$ or $\C$, such as $\C^n$ with the canonical norm $\|x\|_2^2=\sum_{j=1}^n |x_j|^2$ or $\m_{m,n}$, equipped with the Frobenius norm which we denote $\|X\|_F$. All Hilbert spaces over $\C$ are also Hilbert spaces over $\R$ with the scalar product $ \scal{ x,y}_\R=\Re \scal{ x,y}$ and hence it is no restriction to assume that $\V$ is a real Hilbert space wherever needed. Even if $\V$ is a Hilbert space over $\C$ we will implicitly assume that the scalar product is $\scal{ x,y}_\R$.

Given any functional $f:\V\rightarrow \R\cup\{\infty\}$ and parameter $\gamma>0$ we introduce the ``quadratic envelope'' $\Q_\gamma$ as the supremum of all minimizers of the form $\alpha-\frac{\gamma}{2}\|x-y\|^2$ for $\alpha\in\R$ and $y\in\V$;
\begin{equation}\label{QE}
\Q_{\gamma}(f)(x)=\sup_{\alpha\in\R,y\in\V}\left\{\alpha-\frac{\gamma}{2}\|x-y\|^2:~\alpha-\frac{\gamma}{2}\|\cdot-y\|^2\leq f\right\}.
\end{equation}

\begin{wrapfigure}{r}{0.48\textwidth}
     \includegraphics[width=0.48\textwidth]{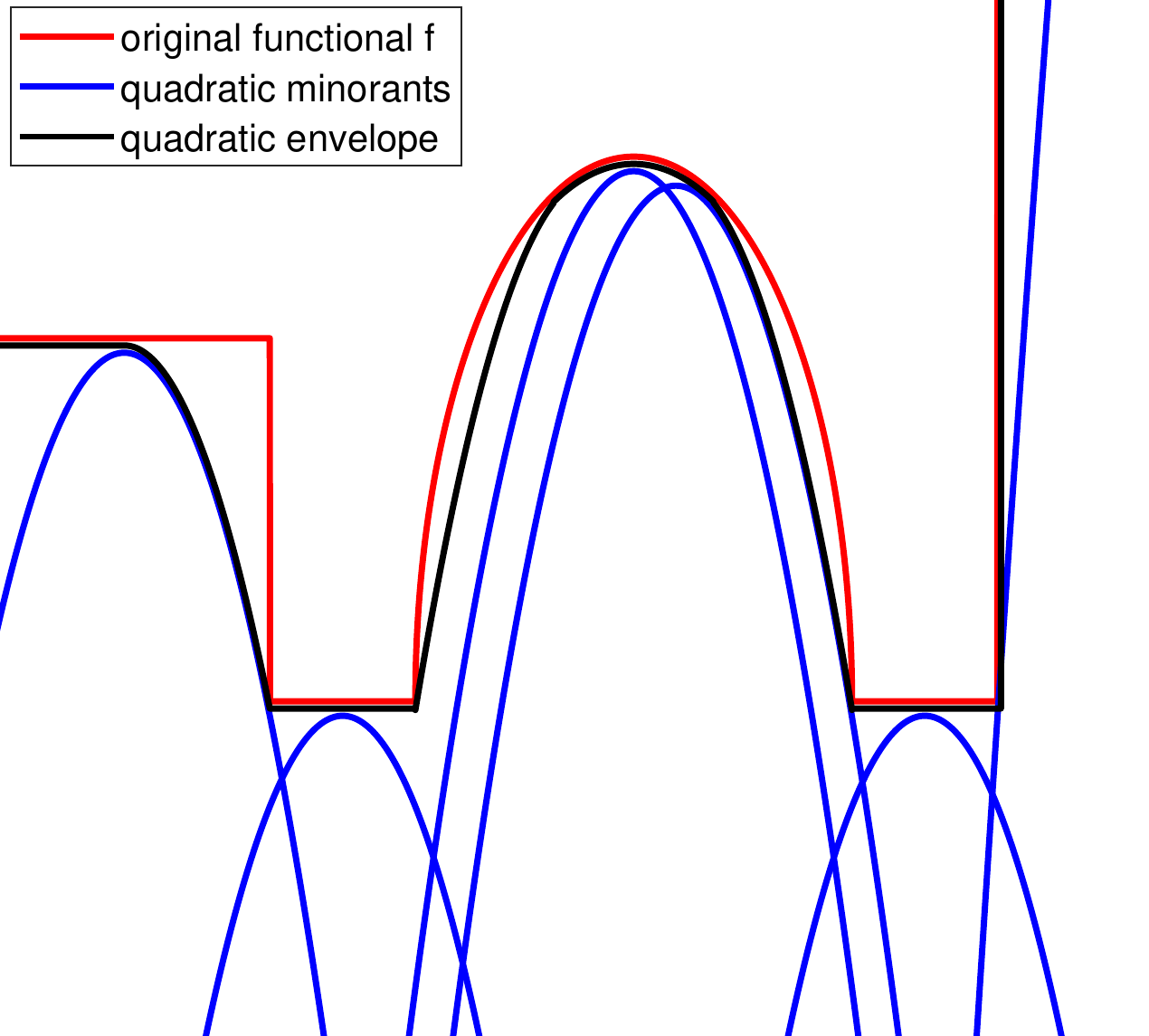}
 \caption{Illustration of a non-convex function $f$ (red) and its quadratic envelope $Q_2(f)$ (black). The black graph lies slightly below for illustration only.}\label{f456}
\end{wrapfigure}

The quadratic envelope has appeared previously e.g.~in \cite{rockafellar2009variational} under the name ``proximal hull'', denoted $h_{\gamma^{-1}}$ (Example 1.44), but it seems that the term is not widespread (see the discussion in Section \ref{sec_rel}) and it seems that its connection with convex envelopes has not been noted or at least not systematically studied. We prefer the term quadratic envelope since it is more illustrative, and prefer the notation $\Q_\gamma$ since it would be messy to always have to invert $\gamma$ which in this context has a concrete meaning; The parameter $\gamma$ basically tunes the maximum negative curvature of $\Q_\gamma(f)$ as we shall see in Section \ref{finer} (Corollary \ref{propconcave}). When $\gamma=1$ we simply write $\Q$ as opposed to $\Q_\gamma$. In this section we first provide some tools to compute $\Q_\gamma$, then prove the connection with l.s.c.~convex envelopes and end with some auxiliary results and a discussion of connections to previous concepts and terminology.

The Legendre transform (or Fenchel conjugate) is defined as $g^*(y):=\sup_x  \scal{ x,y}-g(x).$
We remind the reader that $g^*$ is l.s.c convex and that $g^{**}$ equals the l.s.c.~convex envelope of $g$ by the Fenchel-Moreau theorem (see e.g.~Proposition 13.11 and 13.39 in \cite{bauschke2011convex}).
We now introduce the transform $\S_\gamma$ defined as follows: \begin{equation}\label{defS}\S_\gamma(f)(y):= \para{f(\cdot)+\frac{\gamma}{2}\|\cdot\|^2}^*(\gamma y)-\frac{\gamma}{2}\|y\|^2=\sup_x -f(x)-\frac{\gamma}{2}\fro{x-y}.\end{equation}
$\S_{\gamma}$ is simply the negative of the Moreau envelope computed with constant $\gamma^{-1}$. If we set $q_\gamma(x,y)=-\frac{\gamma}{2}\|x-y\|^2$ then, in the terminology of \cite{rockafellar2009variational} Sec.~11.L, $\S_\gamma(f)$ is the $q_\gamma$-conjugate of $f$ and $\Q_\gamma(f)$ the $q_\gamma$-envelope of $f$ (reinforcing our choice of terminology ``quadratic envelope'' for $\Q_\gamma$). We introduce the symbol $\S_{\gamma}$ mainly since we believe the notation $-e_{\gamma^{-1}}(f)$ or $\mathstrut^{q_{\gamma}}f$ (c.f.~\cite{rockafellar2009variational}) or $-\mathstrut^{\gamma^{-1}}f$ (c.f.~\cite{bauschke2011convex}) would be confusing for our present purposes. Its connection to the quadratic envelope is described by the following proposition;

\begin{proposition}\label{propQ}
Let $\gamma>0$ and let $f$ be a $[0,\infty]$-valued l.s.c.~functional on a separable Hilbert space $\V$. We have $\Q_{\gamma}=\S_\gamma\circ\S_\gamma:=\S_{\gamma}^2$, i.e.~
\begin{equation}\label{lasry1}\Q_{\gamma}(f)(x)=\sup_y\para{\inf_w f(w)+\frac{\gamma}{2}\fro{w-y}}-\frac{\gamma}{2}\fro{x-y} \end{equation}
\end{proposition}
\begin{proof} The argument is a replica of Example 1.44 of \cite{rockafellar2009variational}, but is included for completeness. We have
$\alpha-\frac{\gamma}{2}\|\cdot-y\|^2\leq f$ iff $\alpha\leq f+\frac{\gamma}{2}\|\cdot-y\|^2$ so the maximal $\alpha$ for fixed $y$ is given by $\alpha=-\S_{\gamma}(f)(y)$. Thus $\Q_\gamma(f)(x)=\sup_{y\in\V}-\S_\gamma(f)(y)-\frac{\gamma}{2}\|x-y\|^2=\Q_{\gamma}(f)(x)$ as desired. \end{proof}

The next proposition contains some basic observations on the behavior of $\S_\gamma$ and $\Q_\gamma$.

\begin{proposition}\label{propprop}
Let $\gamma>0$ and let $f$ be a $[0,\infty]$-valued l.s.c.~functional on a separable Hilbert space $\V$. Then $\S_\gamma(f)$ takes values in $(-\infty,0]$ and is continuous whereas $\Q_\gamma(f)$ is lower semi-continuous, takes values in $[0,\infty]$ and is continuous in the interior of $\mathsf{dom}(\Q_\gamma(f))$.
\end{proposition}
\begin{proof}
The statement of the interchanging signs follows easily by the last line of \eqref{defS} which also shows that $\S_\gamma(f)$ avoids $-\infty$. By \eqref{defS} it follows that $\S_\gamma(f)$ (and $\Q_\gamma(f)$ by Proposition \ref{propQ}) is the difference of an l.s.c.~convex functional and a quadratic term. With this in mind the continuity statements follows by standard properties of l.s.c.~convex functionals (see e.g.~Corollary 8.30 \cite{bauschke2011convex}). \end{proof}

The following result is the key result of this section connecting the $\Q_\gamma$-transform with l.s.c.~convex envelopes.

\begin{theorem}\label{t1}
Let $\gamma>0$ and let $f$ be a $[0,\infty]$-valued functional on a separable Hilbert space $\V$. Then
$\para{f+\frac{\gamma}{2}\|\cdot-d\|^2}^*(y)=\S_\gamma(f)\para{\frac{y}{\gamma}+d}+\frac{\gamma}{2}\fro{\frac{y}{\gamma}+d}-\frac{\gamma}{2}\|d\|^2$
and
$$\para{f+\frac{\gamma}{2}\|\cdot-d\|^2}^{**}(x)=\Q_\gamma(f)(x)+\frac{\gamma}{2}\|x-d\|^2.$$
In particular, $\Q_\gamma(f)(x)+\frac{\gamma}{2}\|x-d\|^2$ is the l.s.c.~convex envelope of $f(x)+\frac{\gamma}{2}\|x-d\|^2$ and $0\leq \Q_\gamma(f)\leq f$.
\end{theorem}
\begin{proof}
We have
\begin{align*}
&\para{f(\cdot)+\frac{\gamma}{2}\|\cdot-d\|^2}^*(y)=\sup_x \scal{ x,y}-f(x)-\frac{\gamma}{2}\|x-d\|^2=\\&=\sup_x -f(x)-\frac{\gamma}{2}\fro{x-\para{\frac{y}{\gamma}+d}}+\frac{\gamma}{2}\fro{\frac{y}{\gamma}+d}-\frac{\gamma}{2}\|d\|^2
\end{align*}
from which the first identity follows. Similarly
\begin{align*}
&\para{f(\cdot)+\frac{\gamma}{2}\|\cdot-d\|^2}^{**}(x)=\para{\S_\gamma(f)\para{\frac{\cdot}{\gamma}+d}+\frac{\gamma}{2}\fro{\frac{\cdot}{\gamma}+d}-\frac{\gamma}{2}\|d\|^2}^*(x)\\&=\sup_y  \scal{ x,y}-\S_\gamma(f)\para{\frac{y}{\gamma}+d}-\frac{\gamma}{2}\fro{\frac{y}{\gamma}+d}+\frac{\gamma}{2}\|d\|^2=\\&
=\sup_y  -\S_\gamma(f)\para{\frac{y}{\gamma}+d}-\frac{\gamma}{2}\fro{\frac{y}{\gamma}+d-x}+\frac{\gamma}{2}\|x-d\|^2=\S_\gamma^2(f)(x)+\frac{\gamma}{2}\|x-d\|^2.
\end{align*}
The statement about the convex envelope follows by the Fenchel-Moreau theorem and Proposition \ref{propQ}, which also gives $\Q_\gamma(f)(x)+\frac{\gamma}{2}\|x-d\|^2\leq f(x)+\frac{\gamma}{2}\|x-d\|^2$. This implies the latter part of the inequality $0\leq \Q_\gamma(f)\leq f$ whereas the former has already been noticed in Proposition \ref{propprop}.
\end{proof}

We end this section with some observations about the behavior of $\Q_\gamma(f)$ as a function of $\gamma$.

\begin{proposition}\label{pnew}
Let $f$ be an l.s.c.~$[0,\infty]$-valued functional. Then $\Q_\gamma(f)(x)$ is increasing as a function of $\gamma$. Moreover \begin{equation}\label{ggamma1}\lim_{\gamma\rightarrow \infty}\Q_\gamma(f)(x)=f(x)\end{equation}
whereas the limit as ${\gamma\searrow 0}$ equals a convex minimizer of $f$ above the l.s.c.~convex envelope of $f$.
\end{proposition}

We remark that \eqref{ggamma1} is shown in \cite{rockafellar2009variational}, whereas nothing is said about the case $\gamma\searrow 0$. In fact, $\lim_{\gamma\searrow 0}\Q_\gamma(f)$ usually equals the l.s.c. convex envelope of $f$, but this is not necessarily the case in general, which is a surprise at least for the author. To see this, consider $P=\{x\in\R^2:~x_1>0,~x_2=\sqrt{x_1}\}$, $Q=\{x\in\R^2:~x_1>0,~0<x_2\leq\sqrt{x_1}\}\cup\{0\}$ and $f=\iota_{P}$, where $\iota_P$ is the indicator functional of $P$. It is easy to see that the l.s.c. convex envelope of $\iota_P$ equals $\iota_{cl(Q)}$ (where $cl$ denotes closure) whereas some thinking reveals that $\lim_{\gamma\searrow 0}\Q_\gamma(f)=\iota_Q$. However if $\V$ is finite dimensional and $\lim_{\gamma\searrow 0}\Q_\gamma(f)$ is everywhere finite, then it is automatically continuous (Corollary 8.30 in \cite{bauschke2011convex}), and hence it must equal the l.s.c.~convex envelope of $f$.\newline


\begin{proof}
If $\gamma_1>\gamma_2$ then $\Q_{\gamma_2}(f)(x)+\frac{\gamma_1}{2}\|x\|^2$ equals the l.s.c.~convex functional $\Q_{\gamma_2}(f)(x)+\frac{\gamma_2}{2}\|x\|^2$ plus the term $\frac{\gamma_1-\gamma_2}{2}\|x\|^2$ so it is l.s.c.~and convex. In view of $\Q_{\gamma_2}(f)\leq f$ it also lies below $f+\frac{\gamma_1}{2}\|x\|^2$ and so we conclude that
$$\Q_{\gamma_2}(f)(x)+\frac{\gamma_1}{2}\|x\|^2\leq \big(f+\frac{\gamma_1}{2}\|x\|^2\big)^{**}= \Q_{\gamma_1}(f)(x)+\frac{\gamma_1}{2}\|x\|^2.$$ The first claim follows.
To see \eqref{ggamma1} let $\alpha<f(x)$ be arbitrary. Since $f$ is l.s.c.~the set $\{y:f(y)>\alpha\}$ is open and, as $f\geq 0$, it follows that for any $\gamma$ large enough we have $\alpha-\frac{\gamma}{2}\|\cdot-x\|^2\leq f$. For such $\gamma$ we thus have $\alpha\leq \Q_{\gamma}(f)(x)\leq f(x)$ by \eqref{QE} and Theorem \ref{t1}, so \eqref{ggamma1} follows.


Concerning the limit as ${\gamma\searrow 0}$ set $g(x)=\lim_{\gamma\searrow 0}\Q_\gamma(f)(x)$ which exist by the first part of this proposition. Since $$g(x)=\lim_{\gamma\searrow 0}\Q_\gamma(f)(x)=\lim_{\gamma\searrow 0}\Q_\gamma(f)(x)+\frac{\gamma}{2}\|x\|^2=\lim_{\gamma\searrow 0}\big(f+\frac{\gamma}{2}\|\cdot\|^2\big)^{**}(x)\geq f^{**}$$
we see that $g$ is the limit of a decreasing sequence of convex functions, hence it is also convex (Proposition 8.16 \cite{bauschke2011convex}), and clearly $g\leq f$ by Theorem \ref{t1}.
\end{proof}

\section{Finer Properties of Convex and Quadratic Envelopes}\label{finer}

In this section, we prove a result about the structure of l.s.c.~convex envelopes which seems relatively unknown. For this we need the concept of weak lower semi-continuity, which is nothing but semi-continuity with respect to the weak topology of the underlying separable Hilbert space $\V$. We remind the reader that for convex proper functionals there is no difference (Theorem 9.1 \cite{bauschke2011convex}) between weakly l.s.c.~functionals and standard l.s.c.~functionals. Also, if $\V$ is finite dimensional and the topology is Hausdorff, the two topologies are the same so there is no difference in this case either. However we wish to underline that the difficulty in proving the coming results is present also in the finite-dimensional setting.


We begin with a neat fact concerning weakly l.s.c.~convex envelopes which does not seem to have made its way into the modern literature on the subject. It is a reformulation of Arne Br\o ndsted's extension of Milman's theorem \cite{brondsted1966milman}. To state it we remind the reader that a functional $g$ is \textit{coercive} if and only if its (lower) level sets are bounded (see e.g.~Proposition 11.11 \cite{bauschke2011convex}). Note that l.s.c.~convex envelopes of the type $\Q_\gamma(f)(x)+\frac{\gamma}{2}\|x-d\|^2$ (for positive $f$) always are coercive, by virtue of Proposition \ref{propprop} and the quadratic term.  A function $f$ on $\R$ is called affine if it is of the form $f(t)=at+b$ with $a,b\in\R$.
\begin{theorem}\label{brondsted}
Let $g$ be a weakly l.s.c.~functional on a separable Hilbert space $\V$ such that $g^{**}$ is coercive. Given any $x\in\V$ such that $g(x)\neq g^{**}(x)$ there exists a unit vector $\nu$ and $t_0>0$ such that the function $h(t)= g^{**}(x_0+t\nu)$ is affine on $(-t_0,t_0)$.
\end{theorem}

To prove Theorem \ref{brondsted} we recall some concepts from \cite{brondsted1966milman}. Given a convex function $f$ a point $x$ is called \textit{extremal} if and only if $(x,f(x))$ is extremal for the epigraph of $f$, denoted $[f]$. Equivalently, $x$ is extremal if and only if $x\in\mathsf{dom}~f$ and $f$ is not affine on any relatively open segment containing $x$. Moreover $f_{ext}$ denotes the functional which equals $f(x)$ for all extremal points $x$ and $\infty$ else. As a consequence of Theorem 1 in \cite{brondsted1966milman} we have:

\begin{theorem}\label{arne}
Let $g$ be a weakly l.s.c.~functional on a separable Hilbert space $\V$ such that $g^{**}$ is coercive, then $$[(g^{**})_{ext}]\subset [g].$$
\end{theorem}
\begin{proof}
In the setting of \cite{brondsted1966milman} we let $E$ be the separable Hilbert space $\V$ with the weak topology. Since convex functionals are l.s.c.~with respect to the weak topology if and only if they are with respect to the norm topology it follows that the l.s.c~convex envelope of $g$ equals the weakly l.s.c.~convex envelope. In the notation of Theorem 1 of \cite{brondsted1966milman} we can then take $f=g^{**}$ and the theorem states that $[f_{ext}]\subset [g_{cl}]$ where $g_{cl}$ is the greatest l.s.c.~minorant of $g$. Since $g$ is assumed to be l.s.c.~we have $g=g_{cl}$ and the desired inclusion follows. It remains to check that the conditions of Theorem 1 are fulfilled, which is that ``$g$ is inf-compact in some direction'' (with respect to the weak topology, referring to the terminology of \cite{brondsted1966milman}). For this it suffices to check that $g^{**}$ is inf-compact i.e. that all level sets are compact. The level sets of $g^{**}$ are closed and convex and since $g^{**}$ is assumed coercive they are also bounded. It follows that such level sets are compact in the weak topology and the proof is complete. \end{proof}

Based on this we can now easily prove Theorem \ref{brondsted}.

\noindent\textbf{Proof of Theorem \ref{brondsted}}.
Since $g\geq g^{**}$, Theorem \ref{arne} clearly implies that $g(x)=g^{**}(x)$ for all extremal points $x$ for $g^{**}$. Consequently, if $g(x)= g^{**}(x)$ does not hold, then $x$ is not extremal for $g^{**}$ and the existence of $\nu$ follows by the definition of an extremal point for $g^{**}$.

Next we discuss what the theorem implies about minimizers of $g$ versus $g^{**}$. Denote by $G$ the set of global minimizers of $g$ and by $G^{**}$ the set of global minimizers of $g^{**}$.
\begin{corollary}\label{corminimizers}
Let $g$ be a weakly l.s.c.~functional on a separable Hilbert space $\V$ such that $g^{**}$ is coercive. Then $G^{**}$ is a closed bounded convex set containing $G$. Letting $G^{**}_{ext}$ denote the extremal points of $G^{**}$ we also have that $G^{**}_{ext}\subset G$. Finally the closed convex hull of $G^{**}_{ext}$ equals $G^{**}$.
\end{corollary}
\begin{proof}
The convexity of $G^{**}$ and the inclusion $G\subset G^{**}$ are immediate. The boundedness of $G^{**}$ follows since $g^{**}$ is coercive. Let $x$ be in the closure of $G^{**}$ and let $c$ be the value of the global minimum. Then $g^{**}(x)\leq c$ follows by l.s.c.~and the reverse inequality is obvious from the fact that $c$ is a global minimum. It follows that $x\in G^{**}$ and hence $G^{**}$ is closed.

The existence of points in $G^{**}_{ext}$ and the statement concerning the closed convex hull are now immediate consequences of the Krein-Milman theorem and the fact that bounded closed convex sets are weakly compact in separable Hilbert spaces (Theorem 3.33, \cite{bauschke2011convex}). It remains to prove that $G^{**}_{ext}\subset G$. Let $x_0\in G^{**}_{ext}$ suppose $x_0\not \in G$. Then Theorem \ref{brondsted} implies the existence of a direction $\nu$ on which $g^{**}$ is constant near $x_0$ contradicting that $x_0$ is an extremal point.
\end{proof}

We end by noting that Theorem \ref{brondsted} implies that $\gamma$ tunes the maximum negative curvature in the $\Q_\gamma$-transform as discussed in the introduction.

\begin{corollary}\label{propconcave}
Let $f$ be a weakly l.s.c.~$[0,\infty]$-valued functional on a separable Hilbert space $\V$. For each $x_0\in\V$ with $f(x_0)>\Q_\gamma(f)(x_0)$ there exists a unit vector $\nu$ such that $\Q_\gamma(f)(x_0+t\nu)=a+bt-\frac{\gamma}{2}t^2$ for $t$ near 0 and some $a,b\in\R$.
\end{corollary}
\begin{proof}
Set $g(x)=f(x)+\frac{\gamma}{2}\|x\|^2$. By Theorem \ref{t1} we have $\Q_\gamma(f)(x)+\frac{\gamma}{2}\|x\|^2=g^{**}(x)$ by which it is immediate that $g^{**}$ is coercive (since $\Q_\gamma(f)\geq 0$ by Proposition \ref{propprop}). It also follows that $g(x_0)>g^{**}(x_0)$ and hence Theorem \ref{brondsted} implies that a unit vector $\nu$ exists such that $t\mapsto \Q_\gamma(f)(x+t\nu)$ equals an affine function minus $\frac{\gamma}{2}\|(x+t\nu\|^2$ in a neighborhood of $t=0$.
\end{proof}

\section{The Quadratic Envelope as a Regularizer}\label{Sreg}

We now let $A:\V\rightarrow\W$ be a bounded linear operator, where $\V,\W$ are possibly different (separable) Hilbert spaces, and consider functionals of the type \begin{equation}\label{t4}\J(x)=f(x)+\frac{1}{2}\fro{A x-d}_\W,\quad x\in \V,\end{equation}
Our aim is to develop strategies to deal with the general problem \eqref{t4}, in the case when $f$ is an $[0,\infty]$-valued functional such that $\Q_\gamma(f)$ is computable, and focus on computing (explicit) approximations of the l.s.c convex envelope of $\J$. The theory is split in two cases, either we approximate the convex envelope from below by a convex functional, or we approximate it from above with a non-convex functional having a number of desirable properties, most notably the fact that local minimizers do not change. More precisely, we will study the relationship between the original problem \eqref{t4} and the modified problem
\begin{equation}\label{t4mod}\J_\gamma(x)=\Q_\gamma(f)(x)+\frac{1}{2}\fro{A x-d}_\W,\quad x\in \V\end{equation}
under the assumption that $\gamma I\preccurlyeq A^*A$ or $\gamma I\succcurlyeq A^*A$ (c.f.~\eqref{case 1}-\eqref{case 2} and recall Figure \ref{fig3intro}). Note that $\gamma I\succcurlyeq A^*A$ if and only if $\gamma \geq \|A\|^2$.

\subsection{\textbf{Case $A^*A\succcurlyeq \gamma I$.}}\label{secunderestimate}
Let $f$ be a $[0,\infty]-$valued functional and $A:\V\rightarrow\W$ a bounded linear operator.  The main result of this section states that $\J_\gamma$ is a convex minorant of the l.s.c.~convex envelope $\J^{**}$.

\begin{theorem}\label{tunder}
For $\gamma>0$ such that $A^*A\succcurlyeq \gamma I$, $\J_\gamma$ is convex and $\J_\gamma\leq \J^{**}$. Moreover, if $A^*A\succ \gamma I$ then it is strongly convex, in which case it has a unique minimizer. Finally, a minimizer $\hat x$ of $\J_{\gamma}$ is a minimizer of $\J$ whenever $f(\hat x)=\Q_\gamma(f)(\hat x)$.
\end{theorem}
\begin{proof}
Upon expanding $\fro{Ax-d}=\|Ax\|^2-2\scal{Ax,d}+\fro{d}$ and noting that the latter two terms are affine linear, it is easily seen that it suffices to prove the first part of the statement for $d=0$. That $\J_\gamma$ is l.s.c. and that $\J_\gamma\leq \J$ follows immediately by Theorem \ref{t1} and thus $\J_\gamma\leq \J^{**}$ follows immediately upon showing that $\J_\gamma$ is convex. Define $\scal{x,y}_{\U}=\scal{Ax,Ay}_{\W}-\gamma\scal{x,y}_{\V}$ and note that this is a semi-inner product, as long as $A^*A\succcurlyeq \gamma I$, which is an inner product if the inequality is strict. In either case $\fro{x}_{\U}:=\scal{x,x}_{\U}$ is convex. It follows that $$\Q_\gamma (f)(x)+\frac{1}{2}\fro{Ax}_{\W}=\Big(\Q_\gamma (f)(x)+\frac{\gamma}{2}\fro{x}_{\V}\Big)+\frac{1}{2}\fro{x}_{\U}$$ which by Theorem \ref{t1} implies that $\J_\gamma$ equals the l.s.c. convex envelope of $f(x)+\frac{\gamma}{2}\fro{x}_{\V}$ plus the term $\frac{1}{2}\fro{x}_{\U}$. We conclude that $\J_\gamma$ is a convex functional which is strongly convex when $A^*A\succ \gamma I$. In the latter case the existence of a unique minimizer follows by Corollary 11.15 in \cite{bauschke2011convex} (supercoercivity of $\J_\gamma$ is obvious by the term $\frac{1}{2}\fro{x}_{\U}$). Finally let $d$ be fixed and let $\hat x$ be a minimizer of $\J_\gamma$. Suppose that $f(\hat x)=\Q_\gamma(f)(\hat x)$ and let $y\in\V$ be arbitrary. Then $\J(y)\geq \J_\gamma(y)\geq \J_\gamma(\hat x)=\J(\hat x)$ showing that $\hat x$ is a global minimizer of $\J$.
\end{proof}

\subsection{\textbf{Case $A^*A\preccurlyeq \gamma I$.}}\label{secoverestimate}

Let $f$ be a $[0,\infty]-$valued functional and $A:\V\rightarrow\W$ a bounded linear operator. Again we are interested in the relationship between $\J$ and $\J_\gamma$ defined in \eqref{t4} and \eqref{t4mod} respectively. The main result of this section is that $\J_\gamma$ does not move minima for $\gamma$ in the stated range, but we begin by noting the following inequalities, the first one being reverse of the one proved in Theorem \ref{tunder}.

\begin{proposition}\label{pover}
For $\gamma$ such that $\|A\|^2\leq \gamma$ we have $\J^{**}\leq \J_\gamma\leq \J.$
\end{proposition}
\begin{proof}
The right inequality is immediate since $\Q_\gamma(f)\leq f$ by Theorem \ref{t1}. As in Theorem \ref{tunder} we moreover see that it suffices to prove the left inequality for $d=0$. To this end set $h(x)=\J^{**}(x)-\frac{1}{2}\|Ax\|^2$. Since $\J^{**}\leq f+\frac{1}{2}\fro{Ax}$ we have $h\leq f$ and moreover $$h(x)+\frac{\gamma}{2}\fro{x}=\J^{**}+\para{\frac{\gamma}{2}\fro{x}-\frac{1}{2}\fro{Ax}}.$$ The right hand side is convex and l.s.c.~by which we conclude that $$h(x)+\frac{\gamma}{2}\fro{x}\leq \big(f+\frac{\gamma}{2}\fro{\cdot}\big)^{**}(x)=\Q_\gamma(f)(x)+\frac{\gamma}{2}\fro{x}$$
(the last identity follows by Theorem \ref{t1}) which gives $h(x)\leq \Q_\gamma(f)(x)$. In other words $\J^{**}(x)\leq \Q_\gamma(f)(x)+\frac{1}{2}\|Ax\|^2$ which is the desired inequality (for $d=0$).
\end{proof}

We now come to the main theorem of this section, inspired by Theorems 4.5 and 4.8 in \cite{soubies2015continuous}.  We say that $x$ is a \textit{local minimizer} of $\J$ if there exists a neighborhood $U$ of $x$ in $\V$ such that $\J(y)\geq \J(x)$ for all $y\in U$ and we say that $x$ is a \textit{strict local minimizer} of $\J$ if the inequality is strict for $y\neq x$.

\begin{theorem}\label{tover1}
Suppose that $\|A\|^2< \gamma $. If $x$ is a local minimizer (resp. strict local minimizer) of $\J_\gamma$ then it is also a local minimizer (resp. strict local minimizer) of $\J$, and $\J_\gamma(x)=\J(x)$. In addition the global minimizers coincide.
\end{theorem}
\begin{proof}
Let $x$ be a local minimizer of $\J_\gamma$. If $\Q_\gamma(f)(x)=f(x)$ does not hold then Corollary \ref{propconcave} implies that there exists a unit vector $\nu$ such that \begin{equation}\label{brysel}\frac{d^2}{dt^2}\J_\gamma(x+t\nu)(0)=\frac{d^2}{dt^2}\para{\Q_\gamma(f)(x+t\nu)+\frac{1}{2}\fro{A(x+t\nu)-d}_{\V}}(0)=\|A\nu\|^2-\gamma<0.\end{equation}
We thus conclude that $\Q_\gamma(f)(x)=f(x)$ holds which immediately gives that $\J_\gamma(x)=\J(x)$. In view of Proposition \ref{pover} it follows that $x$ is a local minimizer also for $\J$. The same argument applies to strict local minimizers.

We now prove that the global minimizers coincide. Note that global minimizers of $\J$ are global minimizers of $\J_\gamma$ in view of Proposition \ref{pover} and the fact that $\J(x)=\J^{**}(x)$ for all global minimizers $x$. From this we also see that the global minimum of $\J$ and $\J_\gamma$ coincide, let us denote this value by $c$. Conversely suppose that $x$ is a global minimizer of $\J_\gamma$ (i.e.~$\J_\gamma(x)=c$). Then it is a local minimizer of $\J$ by the first part, which automatically is global for $\J$ since we otherwise would have $\J(y)<c$ for some other value $y$. The proof is complete.
\end{proof}

The situation when $\gamma=\|A\|^2$ is a bit more involved so we content ourselves with the following statement concerning the global minimizers.

\begin{theorem}\label{tover2}
Set $\gamma=\|A\|^2$, let $G$ be the global minimizers of $\J$ and $G_\gamma$ the global minimizers of $\J_\gamma$. Then $G\subset G_\gamma$ and each connected component of $G_\gamma$ contains points of $G$.
\end{theorem}
\begin{proof}
The statement $G\subset G_\gamma$ follows as in the above proof, as well as the fact that the global minimum of $\J$ and $\J_\gamma$ coincide; we denote it by $c$. If $x\in G_\gamma$ and $\J(x)> c$ then it follows by \eqref{brysel} that there exists a unit vector $\nu$ such that $\frac{d^2}{dt^2}\J_\gamma(x+t\nu)\leq 0$ in a neighborhood of $t=0$. Strict inequality contradicts the assumption of global minima, so we deduce that $\gamma\|\nu\|^2=\|A\nu\|^2$. Introducing the semi-norm $\|x\|_{\U}^2=\gamma\|x\|_\V^2-\|Ax\|_\W^2$, this means that $\|\nu\|_{\U}=0$, i.e. that $\nu$ lies in the kernel of the semi-norm $\|\cdot\|_{\U}$ (which is a linear subspace by convexity of the semi-norm). Let $P$ be the affine hyperplane $P=x+\ker \|\cdot\|_{\U}$ and set $S=P\cap G_\gamma.$ For $y\in \ker \|\cdot\|_{\U}$ we have \begin{equation}\label{t6}\J_\gamma(x+y)=\para{\Q_{\gamma}(f)(x+y)+\frac{\gamma}{2}\fro{x+y}_{\V}}-\frac{1}{2}\fro{x}_{\U}- \scal{A(x+y),d}_\W+\frac{1}{2}\|d\|^2_\W,\end{equation} so Theorem \ref{t1} implies that $\J_\gamma$ is convex on $P$. In particular $S$ is convex. Since $\J_\gamma$ is l.s.c.~it is also closed. Moreover $S$ is bounded due to the quadratic term $\fro{x+y}_{\V}$ in \eqref{t6}. $S$ is therefore weakly closed and hence it equals the closed convex hull of its extremal points by the Krein-Milman theorem. If $x$ now is one of these extremal points then we can argue as in the beginning of this proof and conclude that $\J_\gamma(x)=\J(x)$, since the existence of a $\nu$ with the properties stated initially would contradict that $x$ is an extremal point of $S$.
\end{proof}

\section{The $\S$-Transform and Semi-Algebraicity}\label{semi}
We briefly treat semi-algebraicity of $\Q_\gamma(f)$ since it was shown in \cite{attouch2013convergence} that this is a necessary condition for the forward backward splitting method to converge in the non-convex setting. We remind the reader that a function on a finite dimensional space is semi-algebraic if its graph is a semi-algebraic set \cite{bochnak2013real}.
\begin{theorem}\label{tf}
If $\V$ is finite dimensional and $f$ is semi-algebraic then so is $\S_\gamma(f)$ and $\Q_\gamma(f)$.
\end{theorem}
\begin{proof}
We assume for simplicity that $\gamma=1$. It is a consequence of the Tarski-Seidenberg theorem that the set of semi-algebraic functions is closed under addition (see e.g.~Prop.~2.2.6 in \cite{bochnak2013real}) and similarly one can prove that the epigraph of a semi-algebraic function is a semi-algebraic set. If $f$ is semi-algebraic on $\R^n$ it follows that $g(x,y)=\scal{x,y}-(f(x)+\frac{1}{2}\fro{x})$ is semi-algebraic on $\R^{2n}$ and by the argument following Theorem 2.2 in \cite{attouch2013convergence} it follows that the Legendre transform of $f+\frac{1}{2}\fro{x}$ is semi-algebraic. The first result now follows since this function minus $\frac{\gamma}{2}\fro{y}$ equals $\S_\gamma(f)(y)$ by \eqref{defS}, and the second is immediate by Proposition \ref{propQ}.
\end{proof}

\section{Related Works}\label{sec_rel}

The operations $\S_{\gamma}(f)$ and $\Q_\gamma(f)$ were introduced around 1970 in greater generality by J-J.~Moreau \cite{moreau1970inf} and (seemingly independently) E-A.~Weiss \cite{weiss1969konjugierte}, and were further studied around 1990 by R. Poliquin \cite{poliquin1990subgradient} with a focus on smoothness properties. Variations of Propositions \ref{propQ} and \ref{propprop} date back to these early articles, and are also found e.g.~in Rockafellar-Wets \cite{rockafellar2009variational} Section 11.L. The transforms $\S_\gamma$ and $\Q_{\gamma}$ go under names like ``$\Phi$-conjugate''/``proximal transform'' and ``$\Phi$-biconjucate''/``$\Phi$-convex envelope'', and arise by the concrete choice $\Phi(x,y)=q_{\gamma}(x,y)=-\frac{\gamma}{2}\|x-y\|^2$. Following Rockafellar-Wets \cite{rockafellar2009variational} $\Q_\gamma(f)$ should be called ``proximal hull'' or ``$q_\gamma-$envelope''. We believe that the ``quadratic envelope'', which is closer to the latter, is more suggestive. Functions that satisfy $\Q_\gamma(f)=f$ has been called e.g.~$\gamma^{-1}$-proximal or quadratically convex.

However they are called, it seems that the connection with convex envelopes a l\'{a} Theorem \ref{t1} has not been investigated, which is the main novelty of this publication along with the structural result Corollary \ref{propconcave} and its applications to regularization in Section \ref{Sreg}. Apart from the already mentioned works by Aubert, Blanc-Feraud, Soubies and Larsson, Olsson we have not found any similar result in the literature. The fairly recent survey paper \cite{lucet2010shape} is about the closely related concept of computing Fenchel conjugates, and also mentions proximal hulls, yet it has no overlap with the present paper despite citing 262 other papers. It primarily deals with numeric computation of convex envelopes in cases when symbolic formulas are not available, and as such it is an interesting alternative to the methods developed here. The same goes for the papers \cite{mccormick1976computability} and \cite{borwein2009symbolic}. The importance of computing convex envelopes is stressed in \cite{meyer2005convex} where techniques for computing convex envelopes of so called ``convex polyhedral'' functions are developed. Convex approximations from below are considered in \cite{brighi1994approximated} which should be compared with the results in Section \ref{secunderestimate}. An alternative to approximating the convex envelope is to numerically try to compute the proximal operator of the original functional directly, which is pursued in \cite{hare2009computing}. The papers \cite{attouch1993approximation,stromberg1996regularization} deal with Lasry-Lions approximants in Hilbert space but do not make the connection with the convex envelopes. For parameters $s<t$ the Lasry-Lions approximation of $f$ \cite{lasry1986remark} is defined by  \begin{equation}\label{lasry}\begin{aligned}&\S_{1/s}\S_{1/t}(f)(x)=-\para{\inf_y-\para{\inf_w f(w)+\frac{1}{2t}\fro{w-y}}+\frac{1}{2s}\fro{x-y}}=\\& =\sup_y\para{\inf_w f(w)+\frac{1}{2t}\fro{w-y}}-\frac{1}{2s}\fro{x-y} \end{aligned}\end{equation}
which for $s=t$ gives $\Q_{s^{-1}}$. This regularization is also studied in Section 6 of the more recent publication \cite{stromberg1996regularization} (with the notation $C(1)f$), mainly with focus on differentiability-results. It is also closely connected to the more general ``proximal average'', see e.g.~\cite{bauschke2008proximal,hare2009proximal}. However the proximal average has been used mainly for modification of convex functions whereas $\Q_\gamma(f)=f$ for any l.s.c.~convex function.

\section{Conclusions}

We have provided theory for computing l.s.c.~convex envelopes of certain functionals and shown a connection with quadratic envelops (a.k.a.~proximal hulls), which was then used to regularize more intricate problems. We showed that for sufficiently small values of the parameter $\gamma$, this yields convex functionals below the original functional, which coincide with the original functional on a large part of the underlying Hilbert space. For $\gamma$ sufficiently large on the other hand we lose convexity but gain the desirable feature that the modified functional has the same global minimizers as the original one, and fewer local ones. This in turn was based on results regarding the structure of l.s.c.~convex envelopes. The results are inspired from prior work by Carl Olsson and Viktor Larsson as well as Emmanuel Soubies, Laure Blanc-F\'{e}raud and Gilles Aubert.

Particular cases of these ideas have already been applied to compressed sensing, imaging, signal processing and frequency estimation. Currently we are working on more concrete results regarding low rank approximation, improvements of frequency estimation techniques, as well as an application to the classical phase retrieval problem. We hope that other researchers will try these methods on their problems and find that the method is a valuable tool. To aid with this task an expanded version of this article is available on arXiv \cite{carlsson2016arxiv} with many more examples and useful details.

\bibliographystyle{plain}

\bibliography{referenserLL}

\begin{thebibliography}{10}

\bibitem{andersson2016convex}
Fredrik Andersson, Marcus Carlsson, and Carl Olsson.
\newblock Convex envelopes for fixed rank approximation.
\newblock {\em Optimization Letters}, 11(8):1783--1795, 2017.

\bibitem{attouch1993approximation}
H~Attouch and D~Az{\'e}.
\newblock Approximation and regularization of arbitrary functions in {H}ilbert
  spaces by the {L}asry-{L}ions method.
\newblock In {\em Annales de l'IHP Analyse non lin{\'e}aire}, volume~10, pages
  289--312, 1993.

\bibitem{attouch2013convergence}
Hedy Attouch, J{\'e}r{\^o}me Bolte, and Benar~Fux Svaiter.
\newblock Convergence of descent methods for semi-algebraic and tame problems:
  proximal algorithms, forward--backward splitting, and regularized
  {G}auss--{S}eidel methods.
\newblock {\em Mathematical Programming}, 137(1-2):91--129, 2013.

\bibitem{bauschke2011convex}
Heinz~H Bauschke and Patrick~L Combettes.
\newblock {\em Convex analysis and monotone operator theory in Hilbert spaces}.
\newblock Springer Science \& Business Media, 2011.

\bibitem{bauschke2008proximal}
Heinz~H Bauschke, Rafal Goebel, Yves Lucet, and Xianfu Wang.
\newblock The proximal average: basic theory.
\newblock {\em SIAM Journal on Optimization}, 19(2):766--785, 2008.

\bibitem{bochnak2013real}
Jacek Bochnak, Michel Coste, and Marie-Fran{\c{c}}oise Roy.
\newblock {\em Real algebraic geometry}, volume~36.
\newblock Springer Science \& Business Media, 2013.

\bibitem{borwein2009symbolic}
Jonathan~M Borwein and Chris~H Hamilton.
\newblock Symbolic {F}enchel conjugation.
\newblock {\em Mathematical Programming}, 116(1-2):17--35, 2009.

\bibitem{brighi1994approximated}
Bernard Brighi and Michel Chipot.
\newblock Approximated convex envelope of a function.
\newblock {\em SIAM journal on numerical analysis}, 31(1):128--148, 1994.

\bibitem{brondsted1966milman}
Arne Br{\o}ndsted.
\newblock Milman's theorem for convex functions.
\newblock {\em Mathematica Scandinavica}, 19:5--10, 1966.

\bibitem{ctr}
E.~J. Candes, J.~K. Romberg, and T.~Tao.
\newblock Stable signal recovery from incomplete and inaccurate measurements.
\newblock {\em Communications on pure and applied mathematics},
  59(8):1207--1223, 2006.

\bibitem{candes2006robust}
Emmanuel~J Cand{\`e}s, Justin Romberg, and Terence Tao.
\newblock Robust uncertainty principles: Exact signal reconstruction from
  highly incomplete frequency information.
\newblock {\em Information Theory, IEEE Transactions on}, 52(2):489--509, 2006.

\bibitem{candes2005decoding}
Emmanuel~J Candes and Terence Tao.
\newblock Decoding by linear programming.
\newblock {\em IEEE transactions on information theory}, 51(12):4203--4215,
  2005.

\bibitem{carlsson2016arxiv}
Marcus Carlsson.
\newblock On convexification/optimization of functionals including an l2-misfit
  term.
\newblock {\em arXiv preprint arXiv:1609.09378}, 2016.

\bibitem{ourselves}
Marcus Carlsson, Daniele Gerosa, and Carl Olsson.
\newblock A un-biased approach to compressed sensing.
\newblock {\em arXiv preprint}, 2018.

\bibitem{chen2001atomic}
Scott~Shaobing Chen, David~L Donoho, and Michael~A Saunders.
\newblock Atomic decomposition by basis pursuit.
\newblock {\em SIAM review}, 43(1):129--159, 2001.

\bibitem{donoho2006most}
David~L Donoho.
\newblock For most large underdetermined systems of linear equations the
  minimal $\ell^1$-norm solution is also the sparsest solution.
\newblock {\em Communications on pure and applied mathematics}, 59(6):797--829,
  2006.

\bibitem{fazel2002matrix}
Maryam Fazel.
\newblock {\em Matrix rank minimization with applications}.
\newblock PhD thesis, PhD thesis, Stanford University, 2002.

\bibitem{hare2009proximal}
Waren~L Hare.
\newblock A proximal average for nonconvex functions: a proximal stability
  perspective.
\newblock {\em SIAM Journal on Optimization}, 20(2):650--666, 2009.

\bibitem{hare2009computing}
Warren Hare and Claudia Sagastiz{\'a}bal.
\newblock Computing proximal points of nonconvex functions.
\newblock {\em Mathematical Programming}, 116(1-2):221--258, 2009.

\bibitem{larsson2016convex}
Viktor Larsson and Carl Olsson.
\newblock Convex low rank approximation.
\newblock {\em International Journal of Computer Vision}, pages 1--21, 2016.

\bibitem{lasry1986remark}
Jean-Michel Lasry and Pierre-Louis Lions.
\newblock A remark on regularization in {H}ilbert spaces.
\newblock {\em Israel Journal of Mathematics}, 55(3):257--266, 1986.

\bibitem{lucet}
Yves Lucet.
\newblock {The {L}egendre-{F}enchel Transform and the Convex Hull of a
  Function: {F}ast Computational Algorithms, Second-Order Smoothness and
  Analysis}.
\newblock {\em PhD-thesis; https://people.ok.ubc.ca/ylucet/thesis/1997-PhD Yves
  lucet.pdf}, 1997.

\bibitem{lucet2010shape}
Yves Lucet.
\newblock What shape is your conjugate? {A} survey of computational convex
  analysis and its applications.
\newblock {\em SIAM review}, 52(3):505--542, 2010.

\bibitem{mccormick1976computability}
Garth~P McCormick.
\newblock Computability of global solutions to factorable nonconvex programs:
  Part {I} {C}onvex underestimating problems.
\newblock {\em Mathematical programming}, 10(1):147--175, 1976.

\bibitem{meyer2005convex}
Clifford~A Meyer and Christodoulos~A Floudas.
\newblock Convex envelopes for edge-concave functions.
\newblock {\em Mathematical programming}, 103(2):207--224, 2005.

\bibitem{moreau1970inf}
Jean-Jacques Moreau.
\newblock Inf-convolutions, sous-additive, convexite des fonctions numeriques.
\newblock {\em Journal de math{\'e}matiques pures et appliqu{\'e}es},
  49:109--154, 1970.

\bibitem{poliquin1990subgradient}
Ren{\'e}A Poliquin.
\newblock Subgradient monotonicity and convex functions.
\newblock {\em Nonlinear Analysis: Theory, Methods \& Applications},
  14(4):305--317, 1990.

\bibitem{recht2010guaranteed}
Benjamin Recht, Maryam Fazel, and Pablo~A. Parrilo.
\newblock Guaranteed minimum-rank solutions of linear matrix equations via
  nuclear norm minimization.
\newblock {\em SIAM Rev.}, 52(3):471--501, August 2010.

\bibitem{rockafellar2009variational}
R~Tyrrell Rockafellar and Roger J-B Wets.
\newblock {\em Variational analysis}, volume 317.
\newblock Springer, 2009.

\bibitem{soubies2015continuous}
Emmanuel Soubies, Laure Blanc-F{\'e}raud, and Gilles Aubert.
\newblock A continuous exact $\ell_0$ penalty ({CEL0}) for least squares
  regularized problem.
\newblock {\em SIAM Journal on Imaging Sciences}, 8(3):1607--1639, 2015.

\bibitem{stromberg1996regularization}
Thomas Str{\"o}mberg.
\newblock On regularization in banach spaces.
\newblock {\em Arkiv f{\"o}r Matematik}, 34(2):383--406, 1996.

\bibitem{tseng2010approximation}
Paul Tseng.
\newblock Approximation accuracy, gradient methods, and error bound for
  structured convex optimization.
\newblock {\em Mathematical Programming}, 125(2):263--295, 2010.

\bibitem{weiss1969konjugierte}
Ernst-August Weiss.
\newblock Konjugierte funktionen.
\newblock {\em Archiv der Mathematik}, 20(5):538--545, 1969.

\end{thebibliography}

\end{document}